\documentclass[12pt,reqno]{amsart}
\usepackage[hmargin=3cm,vmargin=3cm]{geometry}
\usepackage{amsmath}
\usepackage{amssymb}
\usepackage{amsthm}
\usepackage{enumerate}
\usepackage[pdftex]{color}
\usepackage[pdftex,pdftitle={On the positivity of Riemann-Stieltjes integrals},pdfauthor={Jani Lukkarinen and Mikko S. Pakkanen},colorlinks=TRUE]{hyperref}

\usepackage[pdftex]{graphicx}

\theoremstyle{plain}
\newtheorem{thm}{Theorem}
\newtheorem{que}{Question}

\theoremstyle{remark}
\newtheorem{rem}{Remark}
\theoremstyle{definition}
\newtheorem{exm}{Example}

\newcommand{\ud}{\mathrm{d}}

\newcommand{\eqdefl}{\mathrel{\mathop:}=}
\newcommand{\N}{\mathbb{N}}
\newcommand{\R}{\mathbb{R}}

\newcommand{\myfigure}[2]{ \includegraphics*[#1]{#2} }

\newcommand{\xleft}{x_{\rm L}}

\begin{document}
\title[Positivity of Riemann--Stieltjes integrals]{On the positivity of Riemann--Stieltjes integrals}
\author[J. Lukkarinen]{Jani Lukkarinen} \address{Jani Lukkarinen, Department of Mathematics and Statistics, University of Helsinki, P.O. Box 68, FI-00014 Helsingin yliopisto, Finland.}
\email{jani.lukkarinen@helsinki.fi}
\author[M. S. Pakkanen]{Mikko S. Pakkanen} \address{Mikko S. Pakkanen, CREATES and Department of Economics and Business, Aarhus University, Fuglesangs All\'e 4, DK-8210 Aarhus V, Denmark.}
\email{msp@iki.fi}
\urladdr{\url{http://www.mikkopakkanen.fi/}}
\date{\today}

\begin{abstract}
We study the question, whether a Riemann--Stieltjes integral of a positive continuous function with respect to a non-negative function of bounded variation is positive.
\end{abstract}
\keywords{Riemann--Stieltjes integral, positivity, function of bounded variation, Gr\"onwall's inequality}
\subjclass[2000]{Primary 26A42; Secondary 26A45}
\maketitle

\section{Introduction} \label{sec:intro}

Let $f : [a,b] \rightarrow \R$ be a continuous function and $g : [a,b] \rightarrow \R$ a function of bounded variation. It is a classical result that for such $f$ and $g$, and for any $y \in (a,b]$, the Riemann--Stieltjes integral
\begin{equation}\label{integral}
\int_a^y f(x) \ud g(x),
\end{equation}
exists (see,\ e.g., \cite[pp.~316--317]{ProtterMorrey1977}). While the basic properties of Riemann--Stieltjes integrals (and related Lebesgue--Stieltjes integrals) are covered in classical textbooks on real analysis and integration \cite{HewittStromberg1965,Hildebrandt1963,ProtterMorrey1977,rudin-principles}\footnote{Protter and Morrey \cite{ProtterMorrey1977} have a particularly comprehensive account of Riemann--Stieltjes integrals.},  
the following simple question is not addressed in them.  
\begin{que}\label{question}
If $f$ is positive, $g$ is non-negative, non-vanishing, and satisfies $g(a)=0$, can we then select the upper limit of integration $y$ so that the integral \eqref{integral} is positive?  
\end{que}
The answer to the question is obviously yes for Riemann integrals ($g(x)=x-a$), which could help to explain why it has been overlooked so far.
In fact, after an extensive search of the literature, we believe that Question \ref{question} has not been answered before in the full generality. (We shall comment later, in Remark \ref{finalcomment}, why the present generality may be relevant in applications.)
The only reference known to us is a note by Satyanarayana \cite{Satyanarayana1980} where an affirmative answer is proven under a slightly different set of assumptions, most importantly, assuming that $g$ is \emph{non-decreasing}.

Let us stress that \emph{positive} will here always mean \emph{strictly positive}.  In particular, since $f$ is continuous, under our assumptions $0<\min f\leqslant \max f < \infty$.
Recall that, since $g$ is of bounded variation, there exist non-decreasing functions $g^+$ and $g^-$ such that $g=g^+ - g^-$. Therefore, finite $\lim_{x \rightarrow y^+} g(x)$ and $\lim_{x \rightarrow y^-} g(x)$ exist for any $y \in [a,b]$ (apart from $\lim_{x \rightarrow a^-} g(x)$ and $\lim_{x \rightarrow b^+} g(x)$, obviously).  
Moreover, the integral \eqref{integral} is equal to
\begin{equation*}
\int_a^y f(x) \ud g^+(x) - \int_a^y f(x) \ud g^-(x)\, .
\end{equation*}
Since $g(a)=0$, we may assume that $g^+$ and $g^-$ are non-negative and satisfy $g^\pm(a)=0$.

We may now distinguish two special cases, where it is evident that the answer to Question \ref{question} is yes.
\begin{enumerate}
\item If $\lim_{x \rightarrow \xleft^+} g(x)>0$, where $\xleft  \eqdefl \inf\{ x: g(x) > 0\}$, then there exists $\varepsilon>0$ such that \eqref{integral} is positive for $y=\xleft  + \varepsilon$. 
\item If $g^+(y)>0$ and $g^-(y)=0$, then \eqref{integral} is at least $g^+(y)\min_{x \in [a,y]}f(x)>0$.
\end{enumerate}
The first item follows from the elementary lower bound, valid for $0<\varepsilon\le b-\xleft$,
\begin{equation*}
\int_a^{\xleft+\varepsilon} f(x) \ud g(x)\geqslant g^+(\xleft+\varepsilon)\min_{a,\xleft-\varepsilon\le x\le \xleft+\varepsilon} f(x) -g^-(\xleft+\varepsilon)\max_{a,\xleft-\varepsilon\le x\le \xleft+\varepsilon} f(x),
\end{equation*}
whereas the second item is a straightforward consequence of $g^\pm(a)=0$.

In general, the integral \eqref{integral} is positive if and only if
\begin{equation}\label{variationintegrals}
\int_a^y f(x) \ud g^+(x) > \int_a^y f(x) \ud g^-(x).
\end{equation}
Obviously, the condition $g\geqslant0$ is equivalent to $g^+ \geqslant g^-$, and we have $g(x)>0$ if and only if $g^+(x)>g^-(x)$.  It would be tempting to conjecture that, as a continuous function, $f$ is ``nearly constant" in some neighborhood of $\xleft $ and, hence, that the inequality \eqref{variationintegrals} ought to hold for $y=\xleft  + \varepsilon$ with some ``small" $\varepsilon>0$, suggesting an affirmative answer to Question \ref{question} in general. 

\begin{rem}
The proviso $g(a)=0$ may seem superfluous as the value of the integral \eqref{integral} does not depend on $g(a)$. However, together with the condition $g\geqslant 0$ it constrains the behavior of $g$ near $\xleft$, which is a key part of the formulation of Question \ref{question}. It should also be stressed that the continuity of $f$ is equally important---aside from the possible non-existence of the integral, there is no reason to expect the answer to Question \ref{question} to be yes when $f$ fails to be continuous. 
\end{rem}

\section{Negative answer to Question \ref{question}}

Unfortunately, the heuristic above is too simple-minded since mere continuity does not restrict the \emph{fine properties} of the integrand $f$ and leaves it with enough room to vary ``too much" for our purposes. Indeed, the general answer to Question \ref{question} is no. We show that for any $f$ that exhibits ``enough" variation, there exists a suitable $g$ such that the integral \eqref{integral} is less than zero for all $y \in (a,b]$.

\begin{thm}\label{negative}
Let $f: [a,b] \rightarrow (0,\infty)$ be a continuous function. Suppose that there exist two sequences $(\underline{x}_n)$ and $(\overline{x}_n)$ with
\begin{equation*}
a < \cdots < \underline{x}_n < \overline{x}_n < \cdots <\underline{x}_2 < \overline{x}_2 < \underline{x}_1 < \overline{x}_1 \leqslant b \quad\textrm{and}\quad \lim_{n \rightarrow \infty}\overline{x}_n = a  
\end{equation*}
such that for some $\alpha> 0$ and $\gamma \in (0,1)$, 
\begin{equation*}
f(\overline{x}_n) -f(\underline{x}_n) \geqslant \alpha n^{-\gamma} \quad \text{for all $n \in \N$.}
\end{equation*}
Then, there exists a function $g : [a,b] \rightarrow [0,\infty)$ of bounded variation such that $g(a)=0$ and
\begin{equation*}
\int_a^y f(x) \ud g(x) < 0 \quad \text{for all $y \in (a,b]$}\, .
\end{equation*}
\end{thm}

\begin{proof}
Let $\beta > 1$ and consider $h : [a,b] \rightarrow [0,\infty)$ defined by
\begin{equation*}
h(x) \eqdefl \sum_{n \in \N} n^{-\beta} \chi_{[\underline{x}_n,\overline{x}_n)}(x),
\end{equation*}
where $\chi_E$ denotes the characteristic function of a set $E$. (Figure \ref{fig:unbf} illustrates the definition for Example~\ref{th:negexample} below.) 
This is, by construction, a function of bounded variation such that $h(a)=0$.
We have for any $n \in \N$,
\begin{equation*}
\int_a^{\underline{x}_n}f(x) \ud h(x) = n^{-\beta}f(\underline{x}_n) - \sum_{k = n + 1}^\infty k^{-\beta} \big(f(\overline{x}_k)-f(\underline{x}_k)\big),
\end{equation*}
where
\begin{equation*}
\begin{split}
\sum_{k = n + 1}^\infty k^{-\beta} \big(f(\overline{x}_k)-f(\underline{x}_k)\big) & \geqslant \alpha \sum_{k = n + 1}^\infty k^{-(\beta+\gamma)} \\ & \geqslant \alpha \int_{n+2}^\infty x^{-(\beta+\gamma)} \ud x \\ & = \frac{\alpha}{\beta+\gamma-1} (n+2)^{-(\beta+\gamma-1)}.
\end{split}
\end{equation*}
Since $\gamma < 1$ and $\sup_{n \in \N} f(\underline{x}_n) < \infty$, there exists $n_0 \in \N$ such that
\begin{equation*}
\int_a^{\underline{x}_n}f(x) \ud h(x) < 0 \quad \text{for all $n \geqslant n_0$.}
\end{equation*}
We may also note that for all $n \geqslant n_0$, with $n\geqslant 2$, and $y \in (\underline{x}_{n},\underline{x}_{n-1})$,
\begin{equation*}
\int_a^y f(x) \ud h(x)  \leqslant \int_a^{\underline{x}_{n}}f(x) \ud h(x)< 0.
\end{equation*}
Thus, defining $g \eqdefl h \chi_{[a,\overline{x}_{n_0})}$ yields a function with all the properties stated in the theorem.
\end{proof}

\begin{figure}
  \myfigure{scale=0.76}{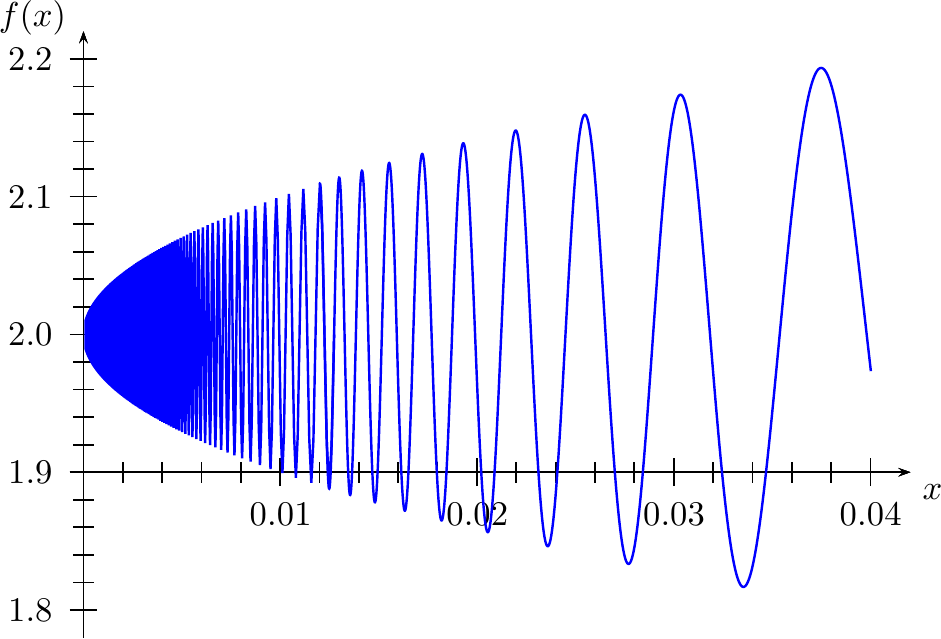} \\
  \myfigure{scale=0.76}{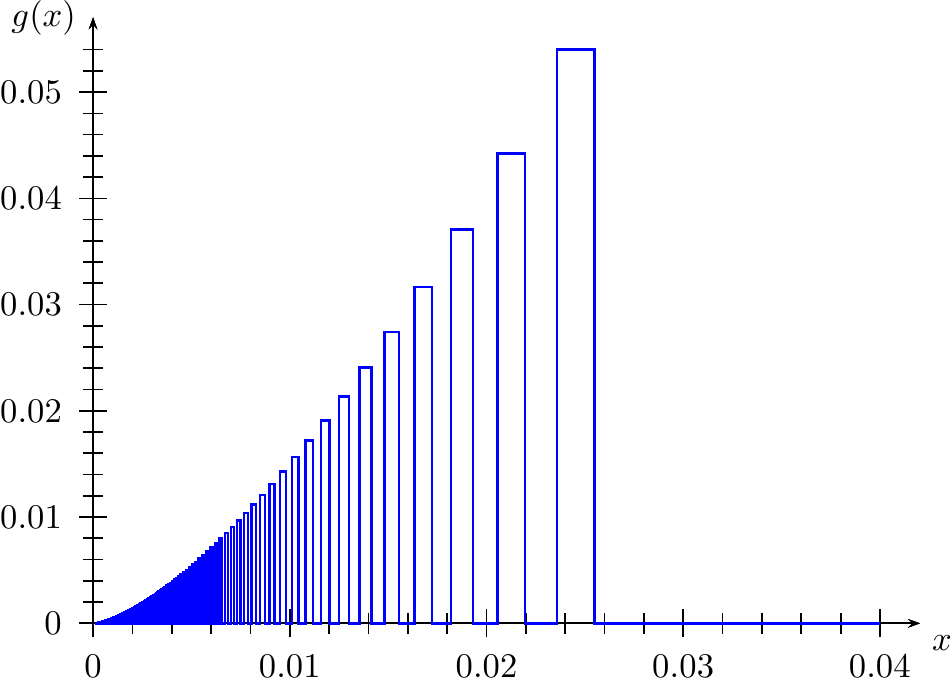}
  \myfigure{scale=0.76}{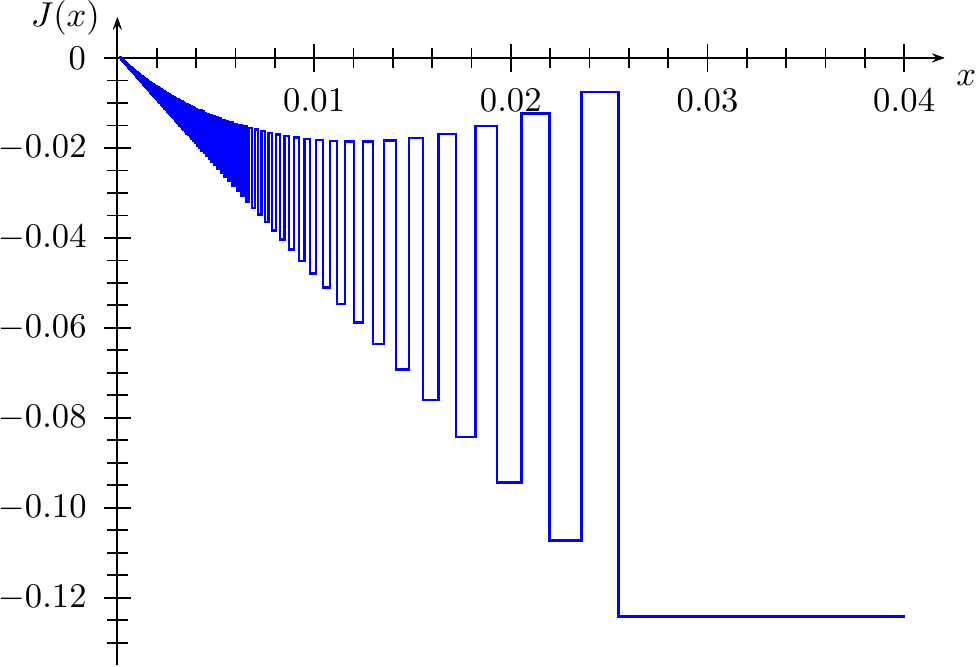}
\caption{Plot of $f$, as defined in Example \ref{th:negexample} for $\gamma=\frac{1}{2}$, and the corresponding functions $g$ and $J$, $J(x) \eqdefl \int_0^x f(x') \ud g(x')$, as defined in the proof of Theorem \ref{negative} using $\beta=\frac{3}{2}$ and $n_0=7$. The values have been computed numerically by considering only contributions with $n\leqslant 1\,000$. For clarity, only the region with $x\in [0,0.04]$ is shown. \label{fig:unbf}}
\end{figure}

\begin{exm}\label{th:negexample}
Function $f : [0,1] \rightarrow (0,\infty)$ given by
\begin{equation*}
f(x) \eqdefl \begin{cases}
x^\gamma \sin(1/x)+2, & \text{if $x > 0$},\\
2 ,& \text{if $x = 0$}, 
\end{cases}
\end{equation*}
where $\gamma \in (0,1)$, satisfies the condition of Theorem \ref{negative} with $\overline{x}_n=\frac{1}{4 n-3} \frac{2}{\pi}$,
$\underline{x}_n=\frac{1}{4 n-1} \frac{2}{\pi}$, $n\in \N$, using $\alpha = 2 (2 \pi)^{-\gamma}$.  Figure \ref{fig:unbf} illustrates the behavior of $f$ and of the corresponding $g$ and Riemann--Stieltjes integral, as defined in the proof of Theorem \ref{negative}. 
\end{exm}

\section{Integrands of bounded variation}

Any integrand $f$ that satisfies the condition of Theorem \ref{negative} is clearly of \emph{unbounded variation}. This prompts us to ask, could we actually obtain an affirmative answer to Question \ref{question} if $f$ varied ``less." In fact, we are able to show that, if $f$ is of \emph{bounded variation}, the answer to Question \ref{question} is yes. Bounds for Riemann--Stieltjes integrals under these assumptions have been derived by Beesack \cite{Beesak75}, Ganelius \cite{Ganelius56}, and Knowles \cite{Knowles84}, but instead of building our argument on them, we give a direct proof which relies on some elementary measure theory and the measure-theoretic version of \emph{Gr\"onwall's inequality}.

\begin{thm}\label{bvcase}
Let $f : [a,b] \rightarrow (0,\infty)$ be a continuous function and $g : [a,b] \rightarrow [0,\infty)$ a non-vanishing function of bounded variation such that $g(a)=0$. If $f$ is of bounded variation, then
\begin{equation*}
\int_a^y f(x) \ud g(x) > 0 \quad \text{for some $y \in (a,b]$.}
\end{equation*}
\end{thm}

\begin{proof}
Let us denote by $V_c^df$ the total variation of $f$ on $[c,d]\subset [a,b]$. 
Recall that total variation is additive in the sense that $V_a^d f = V_a^c f + V_c^d f$ for $a \leqslant c \leqslant d$ \cite[Theorem 12.1]{ProtterMorrey1977}. Moreover,
since $f$ is continuous, the mapping $x \mapsto V_a^x f$ from $[a,b]$ to $[0,\infty)$ is continuous \cite[Theorem 12.2]{ProtterMorrey1977}. Thus, there exists a finite, positive Borel measure $\nu$ on $[a,b]$ such that $\nu([c,d))=V_c^d f$ for any $c$ and $d$ such that $a\leqslant c< d\leqslant b$.
Let us define another finite, positive Borel measure by $\mu(\ud x) \eqdefl f(x)^{-1} \nu(\ud x)$.
By construction, we have then
\begin{equation*}
|f(d)-f(c)| \leqslant \int_{[c,d)} f(x) \mu(\ud x)\, .
\end{equation*}
By an approximation with suitable Riemann--Stieltjes sums, where the values of $g$ are chosen to be sufficiently close to their respective infima on all subintervals of the partitions, we can prove that
 \begin{equation} \label{eq:gdfbound}
\Bigl| \int_a^y g(x) \ud f(x) \Bigr| \leqslant \int_{[a,y)} f(x)g(x) \mu(\ud x)\, .
\end{equation}
The Riemann--Stieltjes integral on the left hand side is well-defined by the 
integration by parts formula \cite[Theorem 12.14]{ProtterMorrey1977}
\begin{equation} \label{eq:ipformula}
\int_a^y g(x) \ud f(x) = f(y)g(y)-f(a)g(a) - \int_a^y f(x) \ud g(x)\, ,
\end{equation}
whenever $\int_a^y f(x) \ud g(x)$ exists and this is always true under the present assumptions.

Now suppose that, contrary to our assertion, we have
\begin{equation}\label{antithesis}
\int_a^y f(x) \ud g(x)\leqslant 0 \quad \text{for all $y \in (a,b]$.}
\end{equation}
Rearranging the integration by parts formula \eqref{eq:ipformula} and using the assumption $g(a)=0$, \eqref{eq:gdfbound},  and \eqref{antithesis}, we obtain for any  $y \in (a,b]$
\begin{equation*}
f(y)g(y) \leqslant \int_a^y g(x) \ud f(x) \leqslant \int_{[a,y)} f(x)g(x) \mu(\ud x)\, .
\end{equation*}
But the measure-theoretic version of Gr\"onwall's inequality \cite[Theorem A.5.1]{EthierKurtz1986} implies that then $f(x)g(x) \leqslant 0$ for all $x \in [a,b]$, whence $g=0$, a contradiction. 
\end{proof}

\begin{rem}
There are two straightforward refinements to Theorem \ref{bvcase}. Firstly, it clearly suffices that $f$ is of bounded variation on $[\xleft ,\xleft +\varepsilon]$ for some $\varepsilon>0$. Secondly, if $g$ is right-continuous, then also the mapping
\begin{equation*}
y \mapsto \int_a^y f(x) \ud g(x)
\end{equation*}
is right-continuous and, under the assumptions of Theorem \ref{bvcase}, there exists an interval $[c,d] \subset [a,b]$ such that
\begin{equation*}
\int_a^y f(x) \ud g(x) > 0 \quad \text{for all $y \in [c,d]$.}
\end{equation*} 
\end{rem}

\begin{rem}\label{finalcomment}
Theorem \ref{bvcase} explains why the answer to Question \ref{question} is perhaps elusive. Experimentation with nicely behaving integrands will not suffice, since a ``pathological" $f$ is required in order to discover the general answer. However, such integrands need not be mere curiosities. In fact, our study of Question \ref{question} was originally motivated by an application in financial mathematics involving as the integrand a \emph{path} of a continuous-time \emph{stochastic process}, which is typically of unbounded variation. 
\end{rem}

\section*{Acknowledgements}

This version corrects a mistake in the formulation of Theorem \ref{negative}, pointed out to us by Gerald Teschl. 
J.~Lukkarinen was supported by the Academy of Finland. M.~S.~Pakkanen acknowledges support from the Finnish Cultural Foundation, from CREATES, funded by the Danish National Research Foundation, and from the Aarhus University Research Foundation regarding the  project ``Stochastic and Econometric Analysis of Commodity Markets".


\begin{thebibliography}{1}

\bibitem{Beesak75}
{P.~R. Beesack},
Bounds for {R}iemann--{S}tieltjes integrals,
{\it Rocky Mountain J. Math.}~{\bf 5} (1975) 75--78. Available at \url{http://dx.doi.org/10.1216/RMJ-1975-5-1-75}.



\bibitem{EthierKurtz1986}
{S.~N. Ethier and T.~G. Kurtz}, {\it Markov Processes:\ Characterization and
  Convergence}, Wiley, New York, 1986.

\bibitem{Ganelius56}
{T.~Ganelius}, Un th\'eor\`eme taub\'erien pour la transformation de Laplace, {\it C. R. Acad. Sci. Paris}~{\bf 242} (1956) 719--721.
  
\bibitem{HewittStromberg1965}
{E.~Hewitt and K.~Stromberg}, {\it Real and Abstract Analysis}, Springer,
  New York, 1965.

\bibitem{Hildebrandt1963}
{T.~H.~Hildebrandt}, {\it Introduction to the Theory of Integration}, Academic Press, New York, 1963.

\bibitem{Knowles84}
{I.~Knowles}, Integral mean value theorems and the Ganelius inequality, {\it Proc. Roy. Soc. Edinburgh Sect. A}~{\bf 97} (1984) 145--150. Available at \url{http://dx.doi.org/10.1017/S0308210500031917}.

\bibitem{ProtterMorrey1977}
{M.~H. Protter and C.~B. Morrey}, {\it A First Course in Real Analysis}, Springer, New York, 1977. 

\bibitem{rudin-principles}
W.~Rudin,
{\it Principles of Mathematical Analysis\/}, McGraw-Hill, New York, 1953.





\bibitem{Satyanarayana1980}
{U.~V. Satyanarayana}, A note on Riemann--Stieltjes integrals, {\it Amer. Math. Monthly}~{\bf 87} (1980)
477--478.  Available at \url{http://www.jstor.org/stable/2320259}.

\end{thebibliography}
\end{document}